\documentclass[12pt]{amsart}
\usepackage{eurosym}
\usepackage{amssymb}
\usepackage{amsmath}
\usepackage{amsfonts}
\usepackage{graphicx,color}

\newtheorem{theorem}{Theorem}
\theoremstyle{plain}

\newtheorem{lemma}{Lemma}

\newtheorem{proposition}{Proposition}
\newtheorem{remark}{Remark}

\numberwithin{equation}{section}


\begin{document}
\title{Multiple solutions to Gierer-Meinhardt systems of elliptic equations}
\author[A. Moussaoui]{Abdelkrim Moussaoui}
\address{Applied Mathematics Laboratory, Faculty of Exact Sciences, \newline
and Biology Department, Faculty of Natural and Life Sciences \newline
A. Mira Bejaia University, 
Targa Ouzemour, 06000 Bejaia, Algeria}
\email{abdelkrim.moussaoui@univ-bejaia.dz}

\begin{abstract}
We establish the existence of multiple solutions for Gierer-Meinhardt system
involving Neumann boundary conditions. The approach combines the methods of
sub-supersolution and Leray-Schauder topological degree.
\end{abstract}

\maketitle

\section{Introduction}

The Gierer-Meinhardt model \cite{GMsyst} proposed in 1972 is a typical model
of reaction-diffusion systems which has been the object of extensive
mathematical treatment in recent years, see \cite{Ni, Wei} for a description
of progress made and references. It describes the activator-inhibitor
coupled behavior for many systems in cell biology and physiology \cite%
{GMsyst, K, Mein}. An activator is a biochemical which stimulates a change
in cells or tissues so that cell differentiation or cell division occurs at
the position where the activator concentration is high. An inhibitor is a
chemical which diffuses much faster than the activator and tempers the
self-enhancing growth of the activator concentration, thereby stabilizing
the system.

The general model proposed by Gierer-Meinhardt \cite{GMsyst} may be written
as%
\begin{equation}
\left\{ 
\begin{array}{l}
u_{t}=d_{1}\Delta u-\hat{d}_{1}u+c\rho \frac{u^{\alpha _{1}}}{v^{\beta _{1}}}%
+\rho _{0}\rho \text{ in }\Omega \times \left[ 0,T\right] \\ 
v_{t}=d_{2}\Delta v-\hat{d}_{2}v+c^{\prime }\rho ^{\prime }\frac{u^{\alpha
_{2}}}{v^{\beta _{2}}}\text{ \ \ \ \ \ \ \ in }\Omega \times \left[ 0,T%
\right] ,%
\end{array}%
\right.
\end{equation}%
where $\Omega $ is bounded domain of $%
\mathbb{R}
^{N}$ ($N\geq 1$) with smooth boundary $\partial \Omega $, under homogeneous
Neumann boundary conditions%
\begin{equation*}
\frac{\partial u}{\partial \eta }=\frac{\partial v}{\partial \eta }=0\text{
\ on}\;\partial \Omega ,
\end{equation*}%
in which $\eta $ denotes the unit outer normal to $\partial \Omega $. The
constants $d_{1},d_{2}$ are diffusion coefficients with $d_{1}\ll d_{2}$ and 
$\hat{d}_{1},\hat{d}_{2},c,c^{\prime }$ and $\rho _{0}$ are positive. The
functions $u(t,x)$ and $v(t,x)$ are the concentrations of an activator
substance and an inhibitor substance, respectively, while the exponents $%
\alpha _{i},\beta _{i}\geq 0$ satisfy the relation $\beta _{1}\alpha
_{2}>\left( \alpha _{1}-1\right) \left( \beta _{2}+1\right) $.

In particular, it has been a matter of high interest to study nonconstant
positive steady states called ground states of Gierer-Meinhardt system. They
are the solutions of the elliptic system%
\begin{equation}
\left\{ 
\begin{array}{ll}
d_{1}\Delta u-\hat{d}_{1}u+\frac{u^{\alpha _{1}}}{v^{\beta _{1}}}=0 & \text{%
in\ }\Omega , \\ 
d_{2}\Delta v-\hat{d}_{2}v+\frac{u^{\alpha _{2}}}{v^{\beta _{2}}}=0 & \text{%
in\ }\Omega ,%
\end{array}%
\right.   \label{4}
\end{equation}%
subject to Neumann boundary conditions. The difficulty in dealing with
problem (\ref{4}) is mainly due to the lack of variational structure and a
priori estimates on the solutions. An idea due to Keener \cite{K} consists
to consider the shadow system associated to (\ref{4}), which is obtained by
dividing by $d_{2}$ in the second equation and then letting $%
d_{2}\rightarrow +\infty $. It has been shown that nonconstant solutions of
the shadow system exhibit interior or boundary spikes. Among the large
number of works in this direction we quote for instance \cite{GG, GWW,
Ni-Takagi, Ni-Takagi2, W1}. The case of finite $d_{2}$ and bounded domain
can be found in \cite{NI-Takagi-Yanagida}, \cite{WW}-\cite{Wei2}, where
existence, stability and/or dynamics of spike solutions are studied.

In whole space $\Omega =%
\mathbb{R}
^{N}$, existence and uniqueness of solutions for a class of Gierer-Meinhardt
system (\ref{4}) are shown in \cite{MKT} for $N\geq 3$. In dimension one or
two ($N=1$,$2$), it has been shown that system (\ref{4}) exhibits single or
multiple bump solutions, see \cite{DKC, DKW} and the references therein.
When the spatial dimension $N=3$, a smoke-ring nonradially solution as well
as a symmetric radially bound solution are constructed for system (\ref{4})
in \cite{KR} and \cite{KWY}, respectively.

The Dirichlet boundary conditions in (\ref{4}) has received a special
attention where existence, nonexistence and uniqueness results have been
shown. Relevant contributions regarding this topic can be found in \cite{CM,
GR, Kim} and the references given there.

Surprisingly enough, so far we were not able to find previous results
providing more than one solution for (\ref{4}), whether in the Dirichlet or
Neumann boundary conditions. Motivated by this fact, our main concern is the
question of existence of multiple ground states for Gierer-Meinhardt system%
\begin{equation*}
(\mathrm{P})\qquad \left\{ 
\begin{array}{ll}
\Delta u-u+\frac{u^{\alpha _{1}}}{v^{\beta _{1}}}=0 & \text{in}\;\Omega , \\ 
\Delta v-v+\frac{u^{\alpha _{2}}}{v^{\beta _{2}}}=0 & \text{in}\;\Omega , \\ 
\frac{\partial u}{\partial \eta }=\frac{\partial v}{\partial \eta }=0 & 
\text{on}\;\partial \Omega ,%
\end{array}%
\right. 
\end{equation*}%
where the exponents $\alpha _{i},\beta _{i}>0$ satisfy%
\begin{equation}
0\leq \alpha _{i}-\beta _{i}<\alpha _{i}+\beta _{i}<1.  \label{alphabeta}
\end{equation}%
It is worth noting that no loss of generality is involved by taking
parameters $d_{i},\hat{d}_{i}$ in (\ref{4}) identically equal to $1$.

By a (weak) solution of system $(\mathrm{P})$ we mean a couple $(u,v)\in
H^{1}(\Omega )\times H^{1}(\Omega )$ such that 
\begin{equation}
\left\{ 
\begin{array}{l}
\int_{\Omega }\nabla u\nabla \varphi \,\mathrm{d}x+\int_{\Omega }u\varphi \,%
\mathrm{d}x=\int_{\Omega }\frac{u^{\alpha _{1}}}{v^{\beta _{1}}}\varphi \,%
\mathrm{d}x, \\ 
\int_{\Omega }\nabla v\nabla \psi \,\mathrm{d}x+\int_{\Omega }v\psi \,%
\mathrm{d}x=\int_{\Omega }\frac{u^{\alpha _{2}}}{v^{\beta _{2}}}\psi \,%
\mathrm{d}x,%
\end{array}%
\right.  \label{defsol}
\end{equation}%
for all $(\varphi ,\psi )\in H^{1}(\Omega )\times H^{1}(\Omega ),$ provided
the integrals in the right-hand side of the above identities exist.

Our main result is stated as follows.

\begin{theorem}
\label{T5}Under assumption (\ref{alphabeta}), problem $(\mathrm{P})$ has at
least three positive solutions in $C^{1,\tau }(\overline{\Omega })\times
C^{1,\tau }(\overline{\Omega }),$ for certain $\tau \in (0,1)$, where at
least one vanishes on $\partial \Omega $ and one is positive on $\partial
\Omega $.
\end{theorem}

The proof is chiefly based on sub-supersolutions techniques and topological
degree theory. It falls naturally into two parts corresponding to the
application of each method.

In the first part, two positive solutions are obtained (cf. Theorems \ref{T1}
and \ref{T3}, section \ref{S2}). They are located in separate areas,
identified by sub-supersolutions pairs, where only one of them is formed by
functions vanishing at the boundary of the domain $\partial \Omega $. The
other area incorporates only positive functions in the entire domain
including the boundary $\partial \Omega $. This is achieved by a choice of
suitable functions with an adjustment of adequate constants on the basis of
which sub-supersolutions pairs are constructed. At this point, spectral
properties of the Laplacian operator have been exploited both in the case of
Dirichlet and Neumann boundary conditions. Then, Theorem \ref{T2}, stated in
Section $2$, ensures the existence of two positive solutions localized in
the aforementioned areas.\ By the nature of the sub-supersolutions pairs
constructed, these solutions do not coincide. One is zero at the boundary of
the domain while the second is positive there. Note that Theorem \ref{T2},
shown via Schauder's fixed point Theorem (\cite{Z}) and suitable truncation,
is a sub-supersolution result for elliptic systems involving Neumann
boundary conditions. It can be applied for large classes of Neumann elliptic
systems since no sign condition on the nonlinearities is required and no
specific structure is imposed, whether cooperative or competitive. For more
inquiries regarding such structures, see, e.g., \cite{DM1, DM2, KM, MM3,
MM2, MM1}.

The second part in this work provides a third solution of $(\mathrm{P})$
stated in Theorem \ref{T5} (cf. section \ref{S4}). The proof is based on
topological degree theory. Precisely, we prove that the degree on a ball $%
\mathcal{B}_{L_{2}}$ comprising both solutions given by Theorem \ref{T1} is
equal to $1$ while the degree in a bigger ball $\mathcal{B}_{L_{1}}\supset 
\mathcal{B}_{L_{2}},$ with $L_{2}<L_{1},$ encompassing all potential
solutions of $(\mathrm{P})$ is $0$. By the excision property of
Leray-Schauder degree, this leads to the existence of a solution for $(%
\mathrm{P})$\ in $\mathcal{B}_{L_{1}}\backslash \overline{\mathcal{B}}%
_{L_{2}},$ lying outside the two merged areas mentioned above where the
first two solutions are located. Therefore, it is a third solution of $(%
\mathrm{P})$.

The rest of this article is organized as follows. Section \ref{S2} treats
the existence of two solutions for system $(\mathrm{P})$ while Section \ref%
{S4} provides the third solution.

\section{Existence of two solutions}

\label{S2}

In the sequel, the Sobolev space $H^{1}(\Omega )$ will be equipped with the
norm 
\begin{equation*}
\Vert w\Vert _{1,2}:=\left( \Vert w\Vert _{2}^{2}+\Vert \nabla w\Vert
_{2}^{2}\right) ^{\frac{1}{2}},\quad w\in H^{1}(\Omega ),
\end{equation*}%
where, as usual, $\Vert w\Vert _{2}:=(\int_{\Omega }|w(x)|^{2}\mathrm{d}x)^{%
\frac{1}{2}}$. We denote by $H_{+}^{1}(\Omega )=\{w\in H^{1}(\Omega ):w\geq
0 $ a.e. in $\Omega \}.$ We also utilize the H\"{o}lder spaces $C^{1}(%
\overline{\Omega })$ and $C^{1,\tau }(\overline{\Omega }),$ $\tau \in (0,1)$
as well as the order cone $\mathcal{C}_{+}^{1}(\overline{\Omega })=\{w\in
C^{1}(\overline{\Omega }):w(x)\geq 0$ for all $x\in \overline{\Omega }\}$.
This cone has a non-empty interior described as follows:%
\begin{equation*}
int\mathcal{C}_{+}^{1}(\overline{\Omega })=\{w\in \mathcal{C}_{+}^{1}(%
\overline{\Omega }):w(x)>0\text{ for all }x\in \overline{\Omega }\}.
\end{equation*}%
In what follows, we set $r^{\pm }:=\max \{\pm r,0\}$ and we denote by $%
\gamma _{0}$ the unique continuous linear map $\gamma _{0}:H^{1}(\Omega
)\rightarrow L^{2}(\partial \Omega )$ known as the trace map such that $%
\gamma _{0}(u)=u|_{\partial \Omega },$ for all $u\in H^{1}(\Omega )\cap C(%
\overline{\Omega })$ and verifies the property $\gamma _{0}(u^{+})=\gamma
_{0}(u)^{+}$ for all $u\in H^{1}(\Omega )$ (see, e.g., \cite{MMP}).

\subsection{A sub-super-solution theorem}

We investigate the existence of solutions to%
\begin{equation*}
(\mathrm{P}_{f,g})\qquad \left\{ 
\begin{array}{ll}
\Delta u-u+f(x,u,v)=0 & \text{in}\;\;\Omega , \\ 
\Delta v-v+g(x,u,v)=0 & \text{in}\;\;\Omega , \\ 
\frac{\partial u}{\partial \eta }=\frac{\partial v}{\partial \eta }=0 & 
\text{on}\;\;\partial \Omega .%
\end{array}%
\right.
\end{equation*}%
where $f,g:\Omega \times \mathbb{R}^{2}\rightarrow \mathbb{R}$ satisfy Carath%
\'{e}odory's conditions.

The following assumptions will be posited.

\begin{itemize}
\item[$(\mathrm{H.1})$] For every $\delta >0,$ there exists $M=M(\delta )>0$
such that 
\begin{equation*}
\max \{|f(x,s_{1},s_{2})|,|g(x,s_{1},s_{2})|\}\leq M,\text{ \ for a.e. }x\in
\Omega \text{, for all }|s_{1}|,|s_{2}|\leq \delta .
\end{equation*}

\item[$(\mathrm{H.2})$] With appropriate $(\underline{u},\underline{v}),(%
\overline{u},\overline{v})\in C^{1}(\overline{\Omega })\times C^{1}(%
\overline{\Omega })$ one has $\underline{u}\leq \overline{u}$, $\underline{v}%
\leq \overline{v}$, as well as 
\begin{equation}
\left\{ 
\begin{array}{l}
\int_{\Omega }(\nabla \underline{u}\nabla \varphi _{1}+\underline{u}\varphi
_{1})\,\mathrm{d}x-\int_{\partial \Omega }\frac{\partial \underline{u}}{%
\partial \eta }\gamma _{0}(\varphi _{1})-\int_{\Omega }f(\cdot ,\underline{u}%
,v)\varphi _{1}\,\mathrm{d}x\leq 0, \\ 
\int_{\Omega }(\nabla \underline{v}\,\nabla \varphi _{2}+\underline{v}%
\varphi _{2})\,\mathrm{d}x-\int_{\partial \Omega }\frac{\partial \underline{v%
}}{\partial \eta }\gamma _{0}(\varphi _{2})-\int_{\Omega }g(\cdot ,u,%
\underline{v})\varphi _{2}\,\mathrm{d}x\leq 0,%
\end{array}%
\right.   \label{c2}
\end{equation}%
\begin{equation}
\left\{ 
\begin{array}{l}
\int_{\Omega }(\nabla \overline{u}\,\nabla \varphi _{1}+\overline{u}\varphi
_{1})\,\mathrm{d}x-\int_{\partial \Omega }\frac{\partial \overline{u}}{%
\partial \eta }\gamma _{0}(\varphi _{1})-\int_{\Omega }f(\cdot ,\overline{u}%
,v)\varphi _{1}\,\mathrm{d}x\geq 0, \\ 
\int_{\Omega }(\nabla \overline{v}\,\nabla \varphi _{2}+\overline{v}\varphi
_{2})\,\mathrm{d}x-\int_{\partial \Omega }\frac{\partial \overline{v}}{%
\partial \eta }\gamma _{0}(\varphi _{2})-\int_{\Omega }g(\cdot ,u,\overline{v%
})\varphi _{2}\,\mathrm{d}x\geq 0%
\end{array}%
\right.   \label{c3}
\end{equation}%
for all $\varphi _{1},\varphi _{2}\in H_{+}^{1}(\Omega )$ and all $(u,v)\in
H^{1}(\Omega )\times H^{1}(\Omega )$ such that $(u,v)\in \lbrack \underline{u%
},\overline{u}]\times \lbrack \underline{v},\overline{v}]$.
\end{itemize}

Under $(\mathrm{H.1})$ the above integrals involving $f$ and $g$ take sense
because $\underline{u},\underline{v},\overline{u},\overline{v}$ are bounded.

\begin{theorem}
\label{T2} Suppose $(\mathrm{H.1})$-$(\mathrm{H.2})$ hold true. Then,
problem $(\mathrm{P}_{f,g})$ possesses a solution $(u,v)\in C^{1,\tau }(%
\overline{\Omega })\times C^{1,\tau }(\overline{\Omega })$ with suitable $%
\tau \in ]0,1[$ such that 
\begin{equation}
\underline{u}\leq u\leq \overline{u}\quad \text{and}\quad \underline{v}\leq
v\leq \overline{v}.  \label{19}
\end{equation}%
Moreover, $\frac{\partial u}{\partial \eta }=\frac{\partial v}{\partial \eta 
}=0$ on $\partial \Omega $.
\end{theorem}

\begin{proof}
Given $(z_{1},z_{2})\in C(\overline{\Omega })\times C(\overline{\Omega })$,
we define 
\begin{equation}
\left\{ 
\begin{array}{l}
\mathrm{T}_{1}(z_{1}):=\min (\max (z_{1},\underline{u}),\overline{u}) \\ 
\mathrm{T}_{2}(z_{2}):=\min (\max (z_{2},\underline{v}),\overline{v}).%
\end{array}%
\right.   \label{0}
\end{equation}%
If $\delta >0$ satisfies 
\begin{equation*}
-\delta \leq \underline{u}\leq \overline{u}\leq \delta ,\quad -\delta \leq 
\underline{v}\leq \overline{v}\leq \delta ,
\end{equation*}%
$(\mathrm{H.1})$ enable us to deduce that 
\begin{equation*}
f(x,\mathrm{T}_{1}(z_{1}),\mathrm{T}_{2}(z_{2})),\text{ }g(x,\mathrm{T}%
_{1}(z_{1}),\mathrm{T}_{2}(z_{2}))\in H^{-1}(\Omega ).
\end{equation*}%
Then, from Minty-Browder Theorem (see, e.g., \cite[Theorem V.15]{B}), we
infer that the auxiliary problem 
\begin{equation}
\left\{ 
\begin{array}{ll}
\Delta u-u+f(x,\mathrm{T}_{1}(z_{1}),\mathrm{T}_{2}(z_{2}))=0 & \text{in }%
\Omega , \\ 
\Delta v-v+g(x,\mathrm{T}_{1}(z_{1}),\mathrm{T}_{2}(z_{2}))=0 & \text{in }%
\Omega , \\ 
\frac{\partial u}{\partial \eta }=\frac{\partial v}{\partial \eta }=0 & 
\text{on }\partial \Omega ,%
\end{array}%
\right.   \label{301}
\end{equation}%
admits a unique solution $(u,v)\in H^{1}(\Omega )\times H^{1}(\Omega ).$

Let us introduce the operator 
\begin{equation*}
\begin{array}{cccc}
\mathcal{T}: & C(\overline{\Omega })\times C(\overline{\Omega }) & 
\rightarrow  & C^{1}(\overline{\Omega })\times C^{1}(\overline{\Omega }) \\ 
& (z_{1},z_{2}) & \mapsto  & (u,v).%
\end{array}%
\end{equation*}%
We note from (\ref{301}) that any fixed point of $\mathcal{T}$ within $[%
\underline{u},\overline{u}]\times \lbrack \underline{v},\overline{v}]$
coincides with the weak solution of $(\mathrm{P}_{f,g})$. Consequently, to
achieve the desired conclusion, we apply Schauder's fixed point Theorem \cite%
{Z} to find a fixed point of $\mathcal{T}$ in $[\underline{u},\overline{u}%
]\times \lbrack \underline{v},\overline{v}]$.

By $(\mathrm{H.1}),$ Moser iteration procedure ensures that $u,v\in
L^{\infty }(\Omega )$ while the regularity theory of Lieberman \cite{L}
implies that $(u,v)\in C^{1,\tau }(\overline{\Omega })\times C^{1,\tau }(%
\overline{\Omega })$ and $\left\Vert u\right\Vert _{C^{1,\tau }(\overline{%
\Omega })},$ $\left\Vert v\right\Vert _{C^{1,\tau }(\overline{\Omega })}\leq
L_{0},$ where $L_{0}>0$ is independent of $u$ and $v$. Then, the compactness
of the embedding $C^{1,\tau }(\overline{\Omega })\subset C(\overline{\Omega }%
)$ implies that $\mathcal{T(}C(\overline{\Omega })\times C(\overline{\Omega }%
))$ is a relatively compact subset of $C(\overline{\Omega })\times C(%
\overline{\Omega })$.

Next, we show that $\mathcal{T}$ is continuous with respect to the topology
of $C(\overline{\Omega })\times C(\overline{\Omega })$. Let $%
(z_{1,n},z_{2,n})\rightarrow (z_{1},z_{2})$ in $C(\overline{\Omega })\times
C(\overline{\Omega })$ for all $n$. Denote $\left( u_{n},v_{n}\right) =%
\mathcal{T(}z_{1,n},z_{2,n})$, which reads as 
\begin{equation}
\begin{array}{c}
\int_{\Omega }(\nabla u_{n}\nabla \varphi _{1}+u_{n}\varphi _{1})\,\mathrm{d}%
x=\int_{\Omega }f(x,\mathrm{T}_{1}(z_{1,n}),\mathrm{T}_{2}(z_{2,n}))\varphi
_{1}\,\mathrm{d}x%
\end{array}
\label{310}
\end{equation}%
and 
\begin{equation}
\begin{array}{c}
\int_{\Omega }(\nabla v_{n}\nabla \varphi _{2}+v_{n}\varphi _{2})\,\mathrm{d}%
x=\int_{\Omega }g(x,\mathrm{T}_{1}(z_{1,n}),\mathrm{T}_{2}(z_{2,n}))\varphi
_{2}\,\mathrm{d}x%
\end{array}
\label{311}
\end{equation}%
for all $\varphi _{1},\varphi _{2}\in H^{1}(\Omega )$. Inserting $(\varphi
_{1},\varphi _{2})=(u_{n},v_{n})$ in (\ref{310}) and (\ref{311}), using $(%
\mathrm{H.1}),$ we get%
\begin{equation}
\left\Vert u_{n}\right\Vert _{1,2}^{2}=\int_{\Omega }f(x,\mathrm{T}%
_{1}(z_{1,n}),\mathrm{T}_{2}(z_{2,n}))u_{n}\,\mathrm{d}x\leq M\int_{\Omega
}u_{n}\,\mathrm{d}x  \label{306}
\end{equation}%
and%
\begin{equation}
\left\Vert v_{n}\right\Vert _{1,2}^{2}=\int_{\Omega }g(x,\mathrm{T}%
_{1}(z_{1,n}),\mathrm{T}_{2}(z_{2,n}))v_{n}\,\mathrm{d}x\leq M\int_{\Omega
}v_{n}\,\mathrm{d}x.  \label{307}
\end{equation}%
Thus, $\{u_{n}\}$ and $\{v_{n}\}$ are bounded in $H^{1}(\Omega )$. So,
passing to relabeled subsequences, we can write the weak convergence in $%
H^{1}(\Omega )\times H^{1}(\Omega )$ 
\begin{equation}
\begin{array}{c}
(u_{n},v_{n})\rightharpoonup \left( u,v\right) 
\end{array}
\label{312}
\end{equation}%
for some $\left( u,v\right) \in H^{1}(\Omega )\times H^{1}(\Omega )$.
Setting $\varphi _{1}=u_{n}-u$ in (\ref{310}) and $\varphi _{2}=v_{n}-v$ in (%
\ref{311}), we find that 
\begin{equation*}
\begin{array}{l}
\int_{\Omega }(\nabla u_{n}\nabla (u_{n}-u)+u_{n}(u_{n}-u))\,\mathrm{d}%
x=\int_{\Omega }f(x,\mathrm{T}_{1}(z_{1,n}),\mathrm{T}%
_{2}(z_{2,n}))(u_{n}-u)\ \mathrm{d}x%
\end{array}%
\end{equation*}%
and 
\begin{equation*}
\begin{array}{l}
\int_{\Omega }(\nabla v_{n}\nabla (v_{n}-v)+v_{n}(v_{n}-v))\text{ }\mathrm{d}%
x=\int_{\Omega }g(x,\mathrm{T}_{1}(z_{1,n}),\mathrm{T}%
_{2}(z_{2,n}))(v_{n}-v)\ \mathrm{d}x.%
\end{array}%
\end{equation*}%
Lebesgue's dominated convergence theorem ensures that%
\begin{equation*}
\begin{array}{c}
\underset{n\rightarrow \infty }{\lim }\left\langle -\Delta
u_{n}+u_{n},u_{n}-u\right\rangle =\underset{n\rightarrow \infty }{\lim }%
\left\langle -\Delta v_{n}+v_{n},v_{n}-v\right\rangle =0.%
\end{array}%
\end{equation*}%
The $S_{+}$-property of $-\Delta $ on $H^{1}\left( \Omega \right) $ (see,
e.g. \cite[Proposition 2.72]{MMP}) along with (\ref{312}) implies%
\begin{equation*}
\begin{array}{c}
(u_{n},v_{n})\rightarrow (u,v)\text{ in }H^{1}(\Omega )\times H^{1}(\Omega ).%
\end{array}%
\end{equation*}%
Then, through (\ref{310}), (\ref{311}) and the invariance of $C(\overline{%
\Omega })\times C(\overline{\Omega })$ by $\mathcal{T}$, we infer that $%
\left( u,v\right) =\mathcal{T(}z_{1},z_{2})$. On the other hand, from (\ref%
{310}) and (\ref{311}) we know that the sequence $\{\left(
u_{n},v_{n}\right) \}$ is bounded in $C^{1,\tau }(\overline{\Omega })\times
C^{1,\tau }(\overline{\Omega })$ for certain $\tau \in (0,1)$ (see \cite[%
Theorem 1.2]{L}). Since the embedding $C^{1,\tau }(\overline{\Omega }%
)\subset C(\overline{\Omega })$ is compact, along a relabeled subsequence
there holds $(u_{n},v_{n})\rightarrow (u,v)$ in $C(\overline{\Omega })\times
C(\overline{\Omega })$. We conclude that $\mathcal{T}$ is continuous.

We are thus in a position to apply Schauder's fixed point theorem to the map 
$\mathcal{T}$, which establishes the existence of $(u,v)\in C^{1}(\overline{%
\Omega })\times C^{1}(\overline{\Omega })$ satisfying $(u,v)=\mathcal{T}%
(u,v).$ By $(\mathrm{H.1})$ and the regularity theory of Lieberman \cite{L}
we derive that $(u,v)\in C^{1,\tau }(\overline{\Omega })\times C^{1,\tau }(%
\overline{\Omega }),$ $\tau \in (0,1)$. Moreover, due to \cite[Theorem 3]{CF}%
, one has 
\begin{equation*}
\frac{\partial u}{\partial \eta }=\frac{\partial v}{\partial \eta }=0\;\;%
\text{on }\partial \Omega .
\end{equation*}%
Hence, $(u,v)$ is a solution of $(\mathrm{P}_{f,g})$. Let us show that (\ref%
{19}) is fulfilled. We only prove the first inequality in (\ref{19}) because
the other ones can justified similarly. Put $\zeta =(\underline{u}-u)^{+}$
and suppose $\zeta \neq 0$. Then, from $(\mathrm{H.2})$, (\ref{301}) and (%
\ref{0}), we get%
\begin{equation*}
\begin{array}{l}
\int_{\{u<\underline{u}\}}(\nabla u\nabla \zeta +u\zeta )\ dx=\int_{\Omega
}(\nabla u\nabla \zeta +u\zeta )\ dx \\ 
=\int_{\Omega }f(x,\mathrm{T}_{1}(u),\mathrm{T}_{2}(v))\zeta \ dx=\int_{\{u<%
\underline{u}\}}f(x,\mathrm{T}_{1}(u),\mathrm{T}_{2}(v))\zeta \ dx \\ 
=\int_{\{u<\underline{u}\}}f(x,\underline{u},\mathrm{T}_{2}(v))\zeta \
dx\geq \int_{\{u<\underline{u}\}}(\nabla \underline{u}\nabla \zeta +%
\underline{u}\zeta )\ dx.%
\end{array}%
\end{equation*}%
This implies that%
\begin{equation*}
\begin{array}{c}
\int_{\{u<\underline{u}\}}(\nabla \underline{u}-\nabla u)\nabla \zeta +(%
\underline{u}-u)\zeta )\ dx\leq 0,%
\end{array}%
\end{equation*}%
a contradiction. Hence $u\geq \underline{u}$ in $\Omega $. A quite similar
argument provides that $v\geq \underline{v},$ $u\leq \overline{u}$ and $%
v\leq \overline{v}$ in $\Omega $. This completes the proof.
\end{proof}

\begin{remark}
The conclusion of Theorem \ref{T2} remains true if we replace Neumann
boundary conditions with Dirichlet ones.
\end{remark}

\subsection{Existence of a solution}

Our goal is to construct sub- and supersolution pair for $(\mathrm{P})$
which, thanks to Theorem \ref{T2}, lead to solutions of $(\mathrm{P})$. With
this aim, consider the following nonlinear Neumann eigenvalue problem%
\begin{equation}
-\Delta \phi _{1}+\phi _{1}=\lambda _{1}\phi _{1}\text{ in }\Omega ,\text{ \ 
}\frac{\partial \phi _{1}}{\partial \eta }=0\text{ on }\partial \Omega .
\label{34}
\end{equation}%
where $\phi _{1}\in int\mathcal{C}_{+}^{1}(\overline{\Omega })$ is the
eigenfunction corresponding to the first eigenvalue $\lambda _{1}>0$ which
is simple, isolate (see \cite{MP}) and characterized by 
\begin{equation}
\lambda _{1}=\inf_{u\in H^{1}(\Omega )\backslash \{0\}}\frac{\int_{\Omega
}(|\nabla u|^{2}+|u|^{2})\,\mathrm{d}x}{\int_{\Omega }|u|^{2}\,\mathrm{d}x}.
\label{71}
\end{equation}%
For a later use, set $\mu ,\overline{\mu }>0$ constants such that 
\begin{equation}
\bar{\mu}=\max_{\overline{\Omega }}\phi _{1}(x)\geq \min_{\overline{\Omega }%
}\phi _{1}(x)=\underline{\mu },\text{ \ }\forall x\in \Omega .  \label{13}
\end{equation}%
Consider the homogeneous Neumann problem 
\begin{equation}
-\Delta z+z=1\text{ in }\Omega ,\text{ \ }\frac{\partial z}{\partial \eta }=0%
\text{ on }\partial \Omega ,  \label{5bis}
\end{equation}%
which, by Minty-Browder Theorem, admits a unique solution $z\in H^{1}(\Omega
)$. A Moser iteration procedure ensures that $z\in L^{\infty }(\Omega )$
while the regularity theory of Lieberman \cite{L} implies that $z\in 
\mathcal{C}_{+}^{1}(\overline{\Omega }).$ Moreover, note that (\ref{5bis})
implies that $\Delta z(x)\leq z(x)$ $a.e.$ in $\Omega .$ Hence, by virtue of
the strong maximum principle of Vasquez \cite{V} we obtain $z(x)>0$ for all $%
x\in \Omega $. Suppose that for some $x_{0}\in \partial \Omega $ we have $%
z(x_{0})=0$. Then, again by \cite{V} it follows that $\frac{\partial z}{%
\partial \eta }(x_{0})<0,$ which contradicts (\ref{5bis}). This proves that $%
z(x)>0$ for all $x\in \overline{\Omega }$, that is, $z\in int\mathcal{C}%
_{+}^{1}(\overline{\Omega })$. Thus, there exists a constant $c_{0}\in (0,1)$
such that 
\begin{equation}
\begin{array}{l}
z\geq c_{0}\bar{\mu}\text{ \ on }\overline{\Omega }.%
\end{array}
\label{11}
\end{equation}%
Furthermore, by comparison principle (see \cite[Lemma 3.2]{ST}), there
exists a constant $c_{1}>0$ such that%
\begin{equation}
z\leq c_{1}\phi _{1}\;\text{in}\;\Omega .  \label{12*}
\end{equation}%
Therefore, it is readily seen that $Cz\geq C^{-1}\phi _{1}$, provided $C>0$
is big enough.

\begin{theorem}
\label{T1} Assume (\ref{alphabeta}) is fulfilled. Then, problem $(\mathrm{P}%
) $ admits a positive solution $(u_{1},v_{1})\in C^{1,\tau }(\overline{%
\Omega })\times C^{1,\tau }(\overline{\Omega }),$ for certain $\tau \in
(0,1) $, within $[C^{-1}\phi _{1},Cz]\times \lbrack C^{-1}\phi _{1},Cz]$,
provided $C>0$ is big enough.
\end{theorem}

\begin{proof}
Let us show that the functions $(\overline{u},\overline{v}):=(Cz,Cz)$
satisfy (\ref{c3}). With this aim, pick $(u,v)\in H^{1}(\Omega )\times
H^{1}(\Omega )$ such that $C^{-1}\phi _{1}\leq u\leq \overline{u}$, $%
C^{-1}\phi _{1}\leq v\leq \overline{v}$. By (\ref{alphabeta}), (\ref{13})
and (\ref{12*}) it follows that%
\begin{equation}
\frac{(Cz)^{\alpha _{1}}}{v^{\beta _{1}}}\leq \frac{(Cz)^{\alpha _{1}}}{%
(C^{-1}\phi _{1})^{\beta _{1}}}\leq C^{\alpha _{1}+\beta _{1}}c_{1}^{\alpha
_{1}}\phi _{1}^{\alpha _{1}-\beta _{1}}\leq C^{\alpha _{1}+\beta
_{1}}c_{1}^{\alpha _{1}}\bar{\mu}^{\alpha _{1}-\beta _{1}}  \label{cla}
\end{equation}%
and%
\begin{equation}
\frac{u^{\alpha _{2}}}{(Cz)^{\beta _{2}}}\leq (Cz)^{\alpha _{2}-\beta
_{2}}\leq (Cc_{1}\phi _{1})^{\alpha _{2}-\beta _{2}}\leq (Cc_{1}\bar{\mu}%
)^{\alpha _{2}-\beta _{2}}.  \label{cla1}
\end{equation}%
Since $\alpha _{i}+\beta _{i}<1$ ($i=1,2$) (see (\ref{alphabeta})), we infer
that%
\begin{equation*}
-\Delta (Cz)+(Cz)=C\geq \max \{C^{\alpha _{1}+\beta _{1}}c_{1}^{\alpha _{1}}%
\bar{\mu}^{\alpha _{1}-\beta _{1}},(Cc_{1}\bar{\mu})^{\alpha _{2}-\beta
_{2}}\}\;\;\text{in}\;\Omega ,
\end{equation*}%
provided that $C>0$ is sufficiently large. Test with $\varphi _{1},\varphi
_{2}\in H_{+}^{1}(\Omega )$ we obtain 
\begin{equation*}
\int_{\Omega }(\nabla (Cz)\nabla \varphi _{1}+Cz\varphi _{1})\,\mathrm{d}%
x\geq \int_{\Omega }\frac{(Cz)^{\alpha _{1}}}{v^{\beta _{1}}}\varphi _{1}\,%
\mathrm{d}x,
\end{equation*}%
\begin{equation*}
\int_{\Omega }(\nabla (Cz)\nabla \varphi _{2}+Cz\varphi _{2})\,\mathrm{d}%
x\geq \int_{\Omega }\frac{u^{\alpha _{2}}}{(Cz)^{\beta _{2}}}\varphi _{2}\,%
\mathrm{d}x,
\end{equation*}%
as desired.

Next, we show that $(\underline{u},\underline{v}):=(C^{-1}\phi
_{1},C^{-1}\phi _{1})$ satisfy (\ref{c2}). In view of (\ref{alphabeta}) and
after increasing $C$ when necessary, one has%
\begin{equation*}
\begin{array}{l}
C^{\alpha _{1}+\beta _{1}-1}\lambda _{1}c_{1}^{\beta _{1}}\bar{\mu}%
\underline{\mu }^{\beta _{1}-\alpha _{1}}\leq 1.%
\end{array}%
\end{equation*}%
Hence, by (\ref{34}), (\ref{13}) and (\ref{12*}), we get%
\begin{equation}
\begin{array}{l}
-\Delta (C^{-1}\phi _{1})+C^{-1}\phi _{1}=C^{-1}\lambda _{1}\phi _{1}\leq
C^{-1}\lambda _{1}\bar{\mu} \\ 
\leq C^{-(\alpha _{1}+\beta _{1})}c_{1}^{-\beta _{1}}\underline{\mu }%
^{\alpha _{1}-\beta _{1}}\leq C^{-(\alpha _{1}+\beta _{1})}c_{1}^{-\beta
_{1}}\phi _{1}^{\alpha _{1}-\beta _{1}} \\ 
\leq \frac{(C^{-1}\phi _{1})^{\alpha _{1}}}{(Cz)^{\beta _{1}}}\leq \frac{%
(C^{-1}\phi _{1})^{\alpha _{1}}}{v^{\beta _{1}}}\text{ \ in }\Omega ,%
\end{array}
\label{46}
\end{equation}%
for all $v\in \lbrack C^{-1}\phi _{1},Cz]$. A quite similar argument based
on (\ref{13}) and (\ref{alphabeta}) furnishes 
\begin{equation}
\begin{array}{l}
-\Delta (C^{-1}\phi _{1})+C^{-1}\phi _{1}=C^{-1}\lambda _{1}\phi _{1}\leq
C^{-1}\lambda _{1}\bar{\mu} \\ 
\leq (C^{-1}\underline{\mu })^{\alpha _{2}-\beta _{2}}\leq (C^{-1}\phi
_{1})^{\alpha _{2}-\beta _{2}} \\ 
\leq \frac{u^{\alpha _{2}}}{(C^{-1}\phi _{1})^{\beta _{2}}}\text{\ \ in }%
\Omega ,\text{ for all }u\in \lbrack C^{-1}\phi _{1},Cz].%
\end{array}
\label{46*}
\end{equation}%
Finally, test (\ref{46})--(\ref{46*}) with $\varphi \in H_{+}^{1}(\Omega )$
we derive that 
\begin{equation*}
\int_{\Omega }(\nabla (C^{-1}\phi _{1})\nabla \varphi +C^{-1}\phi
_{1}\varphi )\,\mathrm{d}x\leq \int_{\Omega }\min \{\frac{(C^{-1}\phi
_{1})^{\alpha _{1}}}{v^{\beta _{1}}},\frac{u^{\alpha _{2}}}{(C^{-1}\phi
_{1})^{\beta _{2}}}\}\varphi \,\mathrm{d}x.
\end{equation*}%
Therefore, the couples $(\underline{u},\overline{u}):=(\underline{v},%
\overline{v}):=(C^{-1}\phi _{1},Cz)$ satisfy assumption $(\mathrm{H.2})$,
whence Theorem \ref{T2} applies leading to existence of a solution $%
(u_{1},v_{1})\in C^{1,\tau }(\overline{\Omega })\times C^{1,\tau }(\overline{%
\Omega }),$ $\tau \in (0,1)$, for problem $(\mathrm{P})$ such that 
\begin{equation}
C^{-1}\phi _{1}\leq u_{1},v_{1}\leq Cz.  \label{22}
\end{equation}%
This ends the proof.
\end{proof}

\subsection{Existence of a second solution}

Consider the homogeneous Dirichlet problem 
\begin{equation}
-\Delta y+y=1\text{ in }\Omega ,\text{ }y=0\text{ on }\partial \Omega ,
\label{10}
\end{equation}%
which admits a unique solution $y\in C_{0}^{1,\tau }(\overline{\Omega })$
for suitable $\tau \in (0,1)$. By a quite similar argument as in the proof
of \cite[Lemma 3.1]{DM}, we show that there exist constants $c>1$ such that 
\begin{equation}
\frac{d}{c}\leq y\leq cd\;\;\text{in}\;\Omega ,\;\;\frac{\partial y}{%
\partial \eta }<0\;\;\text{on}\;\partial \Omega .  \label{12}
\end{equation}%
Now, given $\delta >0$, denote by $y_{\delta }\in C_{0}^{1,\tau }(\overline{%
\Omega })$ the solution of the Dirichlet problem 
\begin{equation}
-\Delta u+u=\left\{ 
\begin{array}{ll}
1 & \text{if }x\in \Omega \backslash \overline{\Omega }_{\delta }, \\ 
-1 & \text{otherwise},%
\end{array}%
\right. \quad u=0\text{ on }\partial \Omega .  \label{1}
\end{equation}%
where 
\begin{equation*}
\Omega _{\delta }:=\{x\in \overline{\Omega }:d(x,\partial \Omega )<\delta \}.
\end{equation*}%
Existence and uniqueness of $y_{\delta }$ directly stem from Minty-Browder's
Theorem. Moreover, according to \cite[Corollary 3.1]{Hai}, if $\delta >0$ is
small enough then%
\begin{equation}
y_{\delta }\geq \frac{y}{2}\text{ in }\Omega \text{ \ \ and \ }\frac{%
\partial y_{\delta }}{\partial \eta }<\frac{1}{2}\frac{\partial y}{\partial
\eta }<0\text{ \ on }\partial \Omega ,  \label{9}
\end{equation}%
Obviously, for a large $C>0,$ one has%
\begin{equation*}
Cy\geq C^{-1}y_{\delta }\text{ in }\overline{\Omega }.
\end{equation*}

\begin{theorem}
\label{T3}Let (\ref{alphabeta}) be satisfied. Then, for $C>0$ big enough,
problem $(\mathrm{P})$ admits a positive solution $(u_{2},v_{2})\in
C^{1,\tau }(\overline{\Omega })\times C^{1,\tau }(\overline{\Omega }),$ $%
\tau \in (0,1)$, within $[C^{-1}y_{\delta },Cy]\times \lbrack
C^{-1}y_{\delta },Cy]$, such that 
\begin{equation*}
u_{2}\neq u_{1}\text{ \ and \ }v_{2}\neq v_{1}.
\end{equation*}
\end{theorem}

\begin{proof}
We shall prove that functions $(\overline{u},\overline{v})=(Cy,Cy)$ and $(%
\underline{u},\underline{v})=(C^{-1}y_{\delta },C^{-1}y_{\delta })$ satisfy (%
\ref{c3}) and (\ref{c2}), respectively. With this aim, let $(u,v)\in
H^{1}(\Omega )\times H^{1}(\Omega )$ such that%
\begin{equation*}
C^{-1}y_{\delta }\leq u,v\leq Cy.
\end{equation*}%
Then, from (\ref{12}), (\ref{9}) and (\ref{alphabeta}), we get%
\begin{equation}
\begin{array}{l}
\frac{v^{\beta _{1}}}{(Cy)^{\alpha _{1}}}\leq (Cy)^{\beta _{1}-\alpha
_{1}}\leq (Ccd(x))^{\beta _{1}-\alpha _{1}}\leq (Cc|\Omega |)^{\beta
_{1}-\alpha _{1}}\text{ in }\Omega%
\end{array}
\label{23}
\end{equation}%
and%
\begin{equation}
\begin{array}{l}
\frac{(Cy_{\delta })^{\beta _{2}}}{u^{\alpha _{2}}}\leq \frac{(Cy)^{\beta
_{2}}}{(C^{-1}y_{\delta })^{\alpha _{2}}}\leq \frac{(Ccd(x))^{\beta _{2}}}{%
(C^{-1}\frac{y}{2})^{\alpha _{2}}}\leq \frac{(Ccd(x))^{\beta _{2}}}{(C^{-1}%
\frac{d(x)}{2c})^{\alpha _{2}}} \\ 
\leq (Cc)^{\beta _{2}+\alpha _{2}}2^{\alpha _{2}}d(x)^{\beta _{2}-\alpha
_{2}}\leq (Cc)^{\beta _{2}+\alpha _{2}}2^{\alpha _{2}}|\Omega |^{\beta
_{2}-\alpha _{2}}\text{ \ in }\Omega .%
\end{array}%
\end{equation}%
On the basis of (\ref{alphabeta}) and for $C>0$ sufficiently large we may
write%
\begin{equation*}
C>\max \{(Cc|\Omega |)^{\beta _{1}-\alpha _{1}},(Cc)^{\beta _{2}+\alpha
_{2}}2^{\alpha _{2}}|\Omega |^{\beta _{2}-\alpha _{2}}\}.
\end{equation*}%
Now, from (\ref{10}), a direct computation gives%
\begin{equation}
-\Delta (Cy)+Cy=C.  \label{21}
\end{equation}%
Then, test with $\varphi \in H_{+}^{1}(\Omega ),$ Green's formula \cite{CF}
together with (\ref{23})-(\ref{21}) imply that%
\begin{equation*}
\begin{array}{c}
\int_{\Omega }(\nabla (Cy)\nabla \varphi +Cy\varphi )\,\mathrm{d}%
x-\int_{\Omega }\frac{\partial (Cy)}{\partial \eta }\gamma _{0}(\varphi )%
\text{ \textrm{d}}\sigma =\int_{\Omega }C\varphi \text{ }\mathrm{d}x \\ 
\geq \int_{\Omega }\max \{\frac{v^{\beta _{1}}}{(Cy)^{\alpha _{1}}},\frac{%
(Cy)^{\beta _{2}}}{u^{\alpha _{2}}}\}\varphi \text{ }\mathrm{d}x,%
\end{array}%
\end{equation*}%
as desired.

Next, we show that $(\underline{u},\underline{v})=(C^{-1}y_{\delta
},C^{-1}y_{\delta })$ satisfy (\ref{c2}). From (\ref{1}) a direct
computation shows that%
\begin{equation}
-\Delta (C^{-1}y_{\delta })+C^{-1}y_{\delta }=\left\{ 
\begin{array}{ll}
C^{-1} & \text{in }\Omega \backslash \overline{\Omega }_{\delta }, \\ 
-C^{-1} & \text{otherwise.}%
\end{array}%
\right.   \label{21*}
\end{equation}%
By (\ref{12}), (\ref{9}) and (\ref{alphabeta}), we get 
\begin{equation}
\begin{array}{l}
\frac{v^{\beta _{1}}}{(C^{-1}y_{\delta })^{\alpha _{1}}}\geq
(C^{-1}y_{\delta })^{\beta _{1}-\alpha _{1}}\geq (C^{-1}\frac{y}{2})^{\beta
_{1}-\alpha _{1}}\geq (\frac{d(x)}{2cC})^{\beta _{1}-\alpha _{1}}\text{ in }%
\Omega 
\end{array}
\label{14}
\end{equation}%
and%
\begin{equation}
\begin{array}{l}
\frac{(C^{-1}y_{\delta })^{\beta _{2}}}{u^{\alpha _{2}}}\geq \frac{%
(C^{-1}y_{\delta })^{\beta _{2}}}{(Cy)^{\alpha _{2}}}\geq \frac{(C^{-1}\frac{%
y}{2})^{\beta _{2}}}{(Ccd(x))^{\alpha _{2}}}\geq \frac{(C^{-1}\frac{d(x)}{2c}%
)^{\beta _{2}}}{(Ccd(x))^{\alpha _{2}}}=\frac{d(x)^{\beta _{2}-\alpha _{2}}}{%
2^{\beta _{2}}(Cc)^{\beta _{2}+\alpha _{2}}}\text{ in }\Omega .%
\end{array}
\label{15}
\end{equation}%
In $\Omega \backslash \overline{\Omega }_{\delta },$ it holds%
\begin{equation*}
(\frac{d(x)}{2cC})^{\beta _{1}-\alpha _{1}}>(\frac{\delta }{2cC})^{\beta
_{1}-\alpha _{1}},
\end{equation*}%
\begin{equation*}
\frac{d(x)^{\beta _{2}-\alpha _{2}}}{2^{\beta _{2}}(Cc)^{\beta _{2}+\alpha
_{2}}}>\frac{\delta ^{\beta _{2}-\alpha _{2}}}{2^{\beta _{2}}(Cc)^{\beta
_{2}+\alpha _{2}}},
\end{equation*}%
because $d(x)>\delta $. Therefore, in view of (\ref{alphabeta}) and for $C>0$
large, we infer that%
\begin{equation}
\min \{(\frac{\delta }{2cC})^{\beta _{1}-\alpha _{1}},\frac{\delta ^{\beta
_{2}-\alpha _{2}}}{2^{\beta _{2}}(Cc)^{\beta _{2}+\alpha _{2}}}\}>C^{-1}%
\text{ \ in }\Omega \backslash \overline{\Omega }_{\delta }\text{.}
\label{26}
\end{equation}

In $\Omega _{\delta },$ it is easily seen that%
\begin{equation}
\min \{(\frac{d(x)}{2cC})^{\beta _{1}-\alpha _{1}},\frac{d(x)^{\beta
_{2}-\alpha _{2}}}{2^{\beta _{2}}(Cc)^{\beta _{2}+\alpha _{2}}}\}\geq
0>-C^{-1}.  \label{26*}
\end{equation}%
Test with $\varphi \in H_{+}^{1}(\Omega )$ in (\ref{21*}), gathering (\ref%
{26}), (\ref{26*}) together and using Green's formula \cite{CF}, we get%
\begin{equation*}
\begin{array}{c}
\int_{\Omega }(\nabla (C^{-1}y_{\delta })\nabla \varphi +C^{-1}y_{\delta
}\varphi )\,\mathrm{d}x-\int_{\Omega }\frac{\partial (C^{-1}y_{\delta })}{%
\partial \eta }\gamma _{0}(\varphi )\text{ \textrm{d}}\sigma  \\ 
=\int_{\Omega \backslash \overline{\Omega }_{\delta }}C^{-1}\varphi \text{ }%
\mathrm{d}x-\int_{\Omega _{\delta }}C^{-1}\varphi \text{ }\mathrm{d}x \\ 
\leq \int_{\Omega }\min \{\frac{v^{\beta _{1}}}{(C^{-1}y_{\delta })^{\alpha
_{1}}},\frac{(C^{-1}y_{\delta })^{\beta _{2}}}{u^{\alpha _{2}}}\}\varphi 
\text{ }\mathrm{d}x.%
\end{array}%
\end{equation*}%
showing that (\ref{c2}) is fulfilled for $(\underline{u},\underline{v}%
):=(C^{-1}y_{\delta },C^{-1}y_{\delta })$.

Consequently, owing to Theorem \ref{T2} we conclude that there exists a
solution $(u_{2},v_{2})\in C^{1,\tau }(\overline{\Omega })\times C^{1,\tau }(%
\overline{\Omega }),$ $\tau \in (0,1)$, of problem $(\mathrm{P})$ such that 
\begin{equation}
C^{-1}y_{\delta }\leq u_{2},v_{2}\leq Cy.  \label{111}
\end{equation}%
On the basis of (\ref{10}), (\ref{1}) and (\ref{111}), we have $%
u_{2},v_{2}=0 $ on $\partial \Omega $ while, in view of (\ref{13}) and (\ref%
{22}), the solution $(u_{1},v_{1})$ in Theorem \ref{T1} satisfies $%
u_{1},v_{1}>0$ on $\partial \Omega $. This shows that $(u_{2},v_{2})$ is a
second positive solution for $(\mathrm{P})$ with $u_{2}\neq u_{1}$ and $%
v_{2}\neq v_{1}$. The proof is completed.
\end{proof}

\section{A third solution}

\label{S4}

In what follows, for any constant $\Lambda >0,$ we denote by $\mathcal{B}%
_{\Lambda }$ the ball in $C^{1}(\overline{\Omega })\times C^{1}(\overline{%
\Omega })$ defined by%
\begin{equation*}
\mathcal{B}_{\Lambda }(0):=\left\{ (u,v)\in C^{1}(\overline{\Omega })\times
C^{1}(\overline{\Omega }):\Vert u\Vert _{C^{1}(\overline{\Omega })}+\Vert
v\Vert _{C^{1}(\overline{\Omega })}<\Lambda \right\} .
\end{equation*}%
Recall from \cite[Definition 1.60]{MMP} that nodal domain of $w\in
H^{1}(\Omega )$ is a connected component of $\Omega \backslash \{x\in \Omega
:w=0\}$. By \cite[Proposition 1.61]{MMP}, if $\Omega _{1}$ is a nodal domain
of $w\in H^{1}(\Omega )\cap C(\Omega )$ then 
\begin{equation}
\chi _{\Omega _{1}}w\in H^{1}(\Omega ),  \label{82}
\end{equation}%
where the symbol $\chi _{\Omega _{1}}$ stands for the characteristic
function of the set $\Omega _{1}$.

\subsection{An auxiliary result}

The next lemma is crucial in finding a third solution for problem $(\mathrm{P%
})$.

\begin{lemma}
\label{L9}For any constant $\lambda \in (0,\lambda _{1})$ the Neumann
problem 
\begin{equation}
-\Delta u+u=\lambda |u^{+}-\phi _{1}|\text{ in }\Omega ,\text{ \ \ }\frac{%
\partial u}{\partial \eta }=0\text{ \ on }\partial \Omega ,  \label{56}
\end{equation}%
does not admit solutions $u\in H^{1}(\Omega )$.
\end{lemma}

\begin{proof}
By contradiction, let $u\in H^{1}(\Omega )$ be a solution of (\ref{56}) from
which it is easily seen that $0$ and $\phi _{1}$ do not fulfill (\ref{56}).
Thus, 
\begin{equation}
u\neq 0\text{ and }u\neq \phi _{1}.  \label{83}
\end{equation}%
First, we show that $u\in L^{\infty }(\Omega )$. To this end, let $\theta :%
\mathbb{R}
\rightarrow \lbrack 0,1]$ be a $C^{1}$ cut-off function such that 
\begin{equation*}
\theta (s)=\left\{ 
\begin{array}{l}
0\text{ if }s\leq 0, \\ 
1\text{ if }s\geq 1%
\end{array}%
\right. \text{ \ and }\phi ^{\prime }(s)\geq 0\text{ in }[0,1].
\end{equation*}%
Fix a constant $L>0$. Here, it does not involve any loss of generality by
assuming that $L>\left\Vert \phi _{1}\right\Vert _{\infty }$. Given $\delta
>0,$ we define $\theta _{\delta }(t)=\theta (\frac{t-L}{\delta })$ for all $%
t\in 
\mathbb{R}
$. It follows that \ 
\begin{equation}
\begin{array}{l}
\theta _{\delta }\circ z\in H^{1}(\Omega )\text{ \ and \ }\nabla (\theta
_{\delta }\circ z)=(\theta _{\delta }^{\prime }\circ z)\nabla z,\text{ \ for 
}z\in H^{1}(\Omega )\text{.}%
\end{array}
\label{5**}
\end{equation}%
Testing in (\ref{56}) with $(\theta _{\delta }\circ u)\varphi $, for $%
\varphi \in H^{1}(\Omega )$ and $\varphi \geq 0$ in $\Omega $, we get%
\begin{equation}
\begin{array}{l}
\int_{\Omega }(\nabla u\nabla ((\theta _{\delta }\circ u)\varphi )+u(\theta
_{\delta }\circ u)\varphi )\,\mathrm{d}x=\lambda \int_{\Omega }|u^{+}-\phi
_{1}|(\theta _{\delta }\circ u)\varphi \,\mathrm{d}x.%
\end{array}
\label{78**}
\end{equation}%
From (\ref{5**}) we have%
\begin{equation*}
\int_{\Omega }\nabla u\nabla ((\theta _{\delta }\circ u)\varphi )\,\mathrm{d}%
x=\int_{\Omega }(\theta _{\delta }^{\prime }\circ u)\varphi \,\mathrm{d}%
x+\int_{\Omega }\nabla u\nabla \varphi \text{ }(\theta _{\delta }\circ u)\,%
\mathrm{d}x.
\end{equation*}%
Since $\theta _{\delta }^{\prime }\circ u\geq 0,$ (\ref{78**}) becomes%
\begin{equation*}
\begin{array}{l}
\int_{\Omega }(\nabla u\nabla \varphi (\theta _{\delta }\circ u)+u(\theta
_{\delta }\circ u)\varphi )\,\mathrm{d}x\leq \lambda \int_{\Omega
}|u^{+}-\phi _{1}|(\theta _{\delta }\circ u)\varphi \,\mathrm{d}x.%
\end{array}%
\end{equation*}%
Letting $\delta \rightarrow 0$ we achieve 
\begin{equation*}
\begin{array}{l}
\int_{\{{u>L}\}}(\nabla {u}\nabla {\varphi }+u\varphi )\text{ }\mathrm{d}%
x\leq \lambda \int_{\{{u>L}\}}|u^{+}-\phi _{1}|\varphi \,\mathrm{d}x,%
\end{array}%
\end{equation*}%
for all $\varphi \in H^{1}(\Omega )$ with $\varphi \geq 0$ in $\Omega $.
Noting that $|u^{+}-\phi _{1}|<u$ for $u>L>0$, it follows that%
\begin{equation}
\begin{array}{l}
\int_{\{{u>L}\}}(\nabla {u}\nabla {\varphi }+u\varphi )\text{ }\mathrm{d}x%
\text{ }\mathrm{d}x\leq \lambda \int_{\{{u>L}\}}u\varphi \ \mathrm{d}x.%
\end{array}
\label{4**}
\end{equation}%
Set $\varphi =u$ in (\ref{4**}) it follows that 
\begin{equation}
\begin{array}{l}
\int_{\{{u>L}\}}(\left\vert \nabla {u}\right\vert ^{2}+u^{2})\text{ }\mathrm{%
d}x\leq \lambda \int_{\{{u>L}\}}u^{2}\ \mathrm{d}x,%
\end{array}%
\end{equation}%
that is,%
\begin{equation}
\begin{array}{l}
\int_{\Omega }(\left\vert \nabla ({u}\chi _{\{u>L\}})\right\vert ^{2}+({u}%
\chi _{\{u>L\}})^{2})\text{ }\mathrm{d}x\leq \lambda \int_{\Omega }(u\chi
_{\{u>L\}})^{2}\ \mathrm{d}x.%
\end{array}
\label{444}
\end{equation}%
By the first eigenvalue's characterization in (\ref{71}) one has%
\begin{equation*}
\begin{array}{l}
\lambda _{1}\int_{\Omega }(u\chi _{\{u>L\}})^{2}\text{ }\mathrm{d}x\leq
\int_{\Omega }(\left\vert \nabla ({u}\chi _{\{u>L\}})\right\vert ^{2}+({u}%
\chi _{\{u>L\}})^{2})\text{ }\mathrm{d}x,%
\end{array}%
\end{equation*}%
while, combined with (\ref{444}), entails 
\begin{equation*}
(\lambda _{1}-\lambda )\int_{\Omega }(\chi _{\{u>L\}}u)^{2}\text{ }\mathrm{d}%
x\leq 0.
\end{equation*}%
Since $\lambda <\lambda _{1,p},$ we conclude that%
\begin{equation*}
\chi _{\{u>L\}}u=0.
\end{equation*}%
whereas the weak comparison principle implies that 
\begin{equation}
u\geq 0\text{ \ in }\Omega .  \label{87**}
\end{equation}%
Thus 
\begin{equation}
0\leq u\leq L\text{ \ in }\Omega ,  \label{99**}
\end{equation}%
showing that $u\in L^{\infty }(\Omega )$ as desired. Consequently, the
regularity result in \cite[Theorem 1.2]{L} implies that $u\in C^{1}(%
\overline{\Omega })$. Moreover, due to (\ref{99**}), the proof of Lemma \ref%
{L9} is completed by showing that the interval $[0,L]$ does not contain any
solution of (\ref{56}). To this end, let us begin by analyzing the case when 
$u>\phi _{1}$. By (\ref{82}) and due to $C^{1}(\overline{\Omega })$%
-regularity of $u$, we can assert that $\chi _{\{u>\phi _{1}\}}u\in
H^{1}(\Omega )$. Testing by $\chi _{\{u>\phi _{1,p}\}}u$ in (\ref{56}) we get%
\begin{equation}
\begin{array}{l}
\int_{\Omega }(|\nabla (\chi _{\{u>\phi _{1}\}}u)|^{2}+(\chi _{\{u>\phi
_{1}\}}u)^{2})\text{ }\mathrm{d}x \\ 
=\int_{\Omega }(\nabla u\nabla (\chi _{\{u>\phi _{1}\}}u)+u\chi _{\{u>\phi
_{1}\}}u)\text{ }\mathrm{d}x \\ 
=\lambda \int_{\Omega }|u^{+}-\phi _{1}|\chi _{\{u>\phi _{1}\}}u\text{ }%
\mathrm{d}x\leq \lambda \int_{\Omega }|u-\phi _{1}|\chi _{\{u>\phi _{1}\}}u%
\text{ }\mathrm{d}x \\ 
\leq \lambda \int_{\Omega }\chi _{\{u>\phi _{1}\}}u^{2}\text{ }\mathrm{d}%
x=\lambda \int_{\Omega }(\chi _{\{u>\phi _{1,p}\}}u)^{2}\text{ }\mathrm{d}x.%
\end{array}
\label{89}
\end{equation}%
Recalling the first eigenvalue's characterization in (\ref{71}) it follows
that%
\begin{equation*}
\begin{array}{l}
\lambda _{1}\int_{\Omega }(\chi _{\{u>\phi _{1}\}}u)^{2}\text{ }\mathrm{d}%
x\leq \int_{\Omega }(|\nabla (\chi _{\{u>\phi _{1}\}}u)|^{2}+(\chi
_{\{u>\phi _{1}\}}u)^{2})\text{ }\mathrm{d}x,%
\end{array}%
\end{equation*}%
which, combined with (\ref{89}) entails 
\begin{equation*}
(\lambda _{1}-\lambda )\int_{\Omega }(\chi _{\{u>\phi _{1,p}\}}u)^{2}\text{ }%
\mathrm{d}x\leq 0.
\end{equation*}%
Since $\lambda <\lambda _{1},$ we infer that%
\begin{equation}
\chi _{\{u>\phi _{1}\}}u=0.  \label{81}
\end{equation}%
Next, we examine the situation when $u<\phi _{1}$ by processing separately
the cases $u\geq \frac{\phi _{1}}{2}$ and $u\leq \frac{\phi _{1}}{2}$.
Notice that 
\begin{equation}
\left\{ 
\begin{array}{l}
|u-\phi _{1}|\leq u\text{ if }u\in \lbrack \frac{\phi _{1}}{2},\phi _{1}] \\ 
\text{and \ }|u-\phi _{1}|\geq u\text{ if }u\in \lbrack 0,\frac{\phi _{1}}{2}%
].%
\end{array}%
\right.   \label{85}
\end{equation}%
Testing (\ref{56}) with $\chi _{\lbrack \frac{\phi _{1}}{2},\phi _{1}]}u\in
H^{1}(\Omega ),$ by (\ref{85}), we get%
\begin{equation}
\begin{array}{l}
\int_{\Omega }|\nabla (\chi _{\lbrack \frac{\phi _{1}}{2},\phi _{1}]}u)|^{2}%
\text{ }\mathrm{d}x=\int_{\Omega }(\nabla u\nabla (\chi _{\lbrack \frac{\phi
_{1}}{2},\phi _{1}]}u)+u\chi _{\lbrack \frac{\phi _{1}}{2},\phi _{1}]}u)%
\text{ }\mathrm{d}x \\ 
=\lambda \int_{\Omega }|u^{+}-\phi _{1}|\chi _{_{\lbrack \frac{\phi _{1}}{2}%
,\phi _{1}]}}u\text{ }\mathrm{d}x\leq \lambda \int_{\Omega }(\chi
_{_{\lbrack \frac{\phi _{1}}{2},\phi _{1}]}}u)^{2}\text{ }\mathrm{d}x.%
\end{array}
\label{84}
\end{equation}%
Then, on account of (\ref{71}), (\ref{84}) becomes%
\begin{equation*}
\lambda _{1}\int_{\Omega }(\chi _{_{\lbrack \frac{\phi _{1}}{2},\phi
_{1}]}}u)^{2}\text{ }\mathrm{d}x\leq \lambda \int_{\Omega }(\chi _{_{\lbrack 
\frac{\phi _{1}}{2},\phi _{1}]}}u)^{2}\text{ }\mathrm{d}x,
\end{equation*}%
leading to 
\begin{equation}
\chi _{_{\lbrack \frac{\phi _{1}}{2},\phi _{1}]}}u=0.  \label{86}
\end{equation}%
Consequently, on the basis of (\ref{83}), (\ref{99**}), (\ref{81}) and (\ref%
{86}), we inevitably must have 
\begin{equation}
\chi _{\lbrack 0,\frac{\phi _{1}}{2}]}u\neq 0.  \label{88}
\end{equation}%
Testing (\ref{56}) with $\chi _{\lbrack 0,\frac{\phi _{1}}{2}]}u\in
H^{1}(\Omega ),$ by (\ref{85}) and (\ref{71}), we obtain%
\begin{equation*}
\begin{array}{l}
\int_{\Omega }|\nabla (\chi _{\lbrack 0,\frac{\phi _{1}}{2}]}u)|^{2}\text{ }%
\mathrm{d}x=\int_{\Omega }|\nabla u\nabla (\chi _{\lbrack 0,\frac{\phi _{1}}{%
2}]}u)+u\chi _{\lbrack 0,\frac{\phi _{1}}{2}]}u)\text{ }\mathrm{d}x \\ 
=\lambda \int_{\Omega }|u-\phi _{1}|\chi _{\lbrack 0,\frac{\phi _{1}}{2}]}u%
\text{ }\mathrm{d}x\geq \lambda \int_{\Omega }(\chi _{\lbrack 0,\frac{\phi
_{1}}{2}]}u)^{2}\text{ }\mathrm{d}x,%
\end{array}%
\end{equation*}%
a contradiction with (\ref{88}) because $\lambda <\lambda _{1}$. Hence, 
\begin{equation}
\chi _{\lbrack 0,\frac{\phi _{1}}{2}]}u=0.  \label{90}
\end{equation}%
Finally, according to (\ref{83}), (\ref{99**}), (\ref{81}), (\ref{86}) and (%
\ref{90}), the desired conclusion follows.
\end{proof}

\subsection{Topological degree results}

\subsubsection{\textbf{Topological degree on }$\mathcal{B}_{L_{1}}(0)$}

For $t\in \lbrack 0,1]$ and for $\lambda \in (0,\lambda _{1})$, we shall
study the homotopy class of problem%
\begin{equation*}
(\mathrm{P}_{t})\qquad \left\{ 
\begin{array}{ll}
-\Delta u+u=t\frac{u^{\alpha _{1}}}{v^{\beta _{1}}}+(1-t)\lambda |u^{+}-\phi
_{1}| & \text{in }\Omega , \\ 
-\Delta v+v=t\frac{u^{\alpha _{2}}}{v^{\beta _{2}}}+(1-t)\lambda |v^{+}-\phi
_{1}| & \text{in }\Omega , \\ 
\frac{\partial u}{\partial \eta }=\frac{\partial v}{\partial \eta }=0\text{
on }\partial \Omega , & 
\end{array}%
\right.
\end{equation*}

The next result shows that solutions of problem $(\mathrm{P}_{t})$ cannot
occur outside the ball $\mathcal{B}_{L_{1}}(0)$ for a large constant $%
L_{1}>0 $.

\begin{lemma}
\label{L1}Assume (\ref{alphabeta}) holds.\ Then, there is a large constant $%
L_{1}>0$ such that every solution $(u,v)$ of $(\mathrm{P}_{t})$ satisfies $%
\left\Vert u\right\Vert _{C^{1}(\overline{\Omega })},\left\Vert v\right\Vert
_{C^{1}(\overline{\Omega })}<L_{1},$ for all $t\in \lbrack 0,1]$. Moreover,
if $t=0$ problem $(\mathrm{P}_{0})$ does not admit solutions.
\end{lemma}

\begin{proof}
By contradiction suppose that for every $n\in 
\mathbb{N}
,$ there exist $t_{n}\in \lbrack 0,1]$ and a solution $(u_{n},v_{n})$ of $(%
\mathrm{P}_{t_{n}})$ such that%
\begin{equation*}
t_{n}\rightarrow t\in \lbrack 0,1]\text{ \ and \ }\Vert u_{n}\Vert _{C^{1}(%
\overline{\Omega })},\Vert v_{n}\Vert _{C^{1}(\overline{\Omega }%
)}\rightarrow \infty \text{ \ as }n\rightarrow \infty .
\end{equation*}%
Thus, $(\mathrm{P}_{t_{n}})$ is reads as%
\begin{equation*}
\left\{ 
\begin{array}{l}
-\Delta u_{n}+u_{n}=t_{n}\frac{u_{n}^{\alpha _{1}}}{v_{n}^{\beta _{1}}}%
+(1-t_{n})\lambda |u_{n}^{+}-\phi _{1}|\text{ in }\Omega , \\ 
-\Delta v_{n}+v_{n}=t_{n}\frac{u_{n}^{\alpha _{2}}}{v_{n}^{\beta _{2}}}%
+(1-t_{n})\lambda |v_{n}^{+}-\phi _{1}|\text{ in }\Omega , \\ 
\frac{\partial u_{n}}{\partial \eta }=\frac{\partial v_{n}}{\partial \eta }=0%
\text{ on }\partial \Omega ,\text{ \ for all }n\in 
\mathbb{N}
.%
\end{array}%
\right. 
\end{equation*}%
Set%
\begin{equation}
\begin{array}{c}
(\theta _{n},\hat{\theta}_{n}):=(\Vert u_{n}\Vert _{C^{1}(\overline{\Omega }%
)},\Vert v_{n}\Vert _{C^{1}(\overline{\Omega })})\rightarrow \infty \text{
as }n\rightarrow \infty 
\end{array}
\label{3}
\end{equation}%
and denote 
\begin{equation}
\left\{ 
\begin{array}{l}
(\mathcal{U}_{n},\mathcal{V}_{n}):=(\frac{u_{n}}{\theta _{n}},\frac{v_{n}}{%
\hat{\theta}_{n}})\in C^{1}(\overline{\Omega })\times C^{1}(\overline{\Omega 
}), \\ 
\text{with }\Vert \mathcal{U}_{n}\Vert _{C^{1}(\overline{\Omega })}=\Vert 
\mathcal{V}_{n}\Vert _{C^{1}(\overline{\Omega })}=1\text{ for all }n\in 
\mathbb{N}
.%
\end{array}%
\right.   \label{25}
\end{equation}%
$(\mathrm{P}_{t_{n}})$ results in 
\begin{equation}
\left\{ 
\begin{array}{l}
-\Delta \mathcal{U}_{n}+\mathcal{U}_{n}=\frac{1}{\theta _{n}}\left( t_{n}%
\frac{u_{n}^{\alpha _{1}}}{v_{n}^{\beta _{1}}}+(1-t_{n})\lambda
|u_{n}^{+}-\phi _{1}|\right) \text{ in }\Omega , \\ 
-\Delta \mathcal{V}_{n}+\mathcal{V}_{n}=\frac{1}{\hat{\theta}_{n}}\left(
t_{n}\frac{u_{n}^{\alpha _{2}}}{v_{n}^{\beta _{2}}}+(1-t_{n})\lambda
|v_{n}^{+}-\phi _{1}|\right) \text{ in }\Omega , \\ 
\frac{\partial \mathcal{U}_{n}}{\partial \eta }=\frac{\partial \mathcal{V}%
_{n}}{\partial \eta }=0\text{ on }\partial \Omega .%
\end{array}%
\right.   \label{13*}
\end{equation}%
Note that (\ref{13*}) implies that%
\begin{equation*}
\Delta \mathcal{U}_{n}(x)\leq \mathcal{U}_{n}(x)\text{ \ and \ }\Delta 
\mathcal{V}_{n}(x)\leq \mathcal{V}_{n}(x)\text{ for a.e. }x\in \Omega .
\end{equation*}%
Hence, by virtue of the strong maximum principle in \cite{V}, we obtain $%
\mathcal{U}_{n}(x),\mathcal{V}_{n}(x)>0$ for all $x\in \Omega $. Suppose
that for some $x_{0},x_{0}^{\prime }\in \partial \Omega $ we have $\mathcal{U%
}_{n}(x_{0})=\mathcal{V}_{n}(x_{0}^{\prime })=0$. Then, again by \cite{V} it
follows that $\frac{\partial \mathcal{U}_{n}}{\partial \eta }(x_{0}),\frac{%
\partial \mathcal{V}_{n}}{\partial \eta }(x_{0}^{\prime })<0,$ which
contradicts (\ref{13}). This proves that $\mathcal{U}_{n}(x),\mathcal{V}%
_{n}(x)>0$ for all $x\in \overline{\Omega }$, that is, $\mathcal{U}_{n},%
\mathcal{V}_{n}\in int\mathcal{C}_{+}^{1}(\overline{\Omega })$. Thus, there
exists a constant $\rho >0$ such that 
\begin{equation}
\begin{array}{l}
\mathcal{U}_{n},\mathcal{V}_{n}>\rho \text{ \ on }\overline{\Omega }.%
\end{array}
\label{5}
\end{equation}%
By (\ref{5}), (\ref{alphabeta}), (\ref{25}) and in view of (\ref{3}), one has%
\begin{equation}
\begin{array}{c}
\left\vert \frac{1}{\theta _{n}}\left( t_{n}\frac{u_{n}^{\alpha _{1}}}{%
v_{n}^{\beta _{1}}}+(1-t_{n})\lambda |u_{n}^{+}-\phi _{1}|\right)
\right\vert  \\ 
=\frac{t_{n}}{\theta _{n}^{1+\beta _{1}-\alpha _{1}}}\frac{\mathcal{U}%
_{n}^{\alpha _{1}}}{\mathcal{V}_{n}^{\beta _{1}}}+(1-t_{n})\lambda |\mathcal{%
U}_{n}-\frac{\phi _{1}}{\theta _{n}}| \\ 
\leq \frac{t_{n}}{\theta _{n}^{1+\beta _{1}-\alpha _{1}}\rho ^{\beta _{1}}}%
\mathcal{U}_{n}^{\alpha _{1}}+(1-t_{n})\lambda \mathcal{U}_{n}\leq \mathcal{U%
}_{n}^{\alpha _{1}}+\lambda \mathcal{U}_{n} \\ 
\leq \Vert \mathcal{U}_{n}\Vert _{C^{1}(\overline{\Omega })}^{\alpha
_{1}}+\lambda \Vert \mathcal{U}_{n}\Vert _{C^{1}(\overline{\Omega }%
)}=1+\lambda \text{ \ in }\Omega ,%
\end{array}
\label{16}
\end{equation}%
and%
\begin{equation}
\begin{array}{l}
\left\vert \frac{1}{\hat{\theta}_{n}}\left( t_{n}\frac{u_{n}^{\alpha _{2}}}{%
v_{n}^{\beta _{2}}}+(1-t_{n})\lambda |v_{n}^{+}-\phi _{1}|\right)
\right\vert  \\ 
=\frac{t_{n}}{\hat{\theta}_{n}^{1+\beta _{2}-\alpha _{2}}}\frac{\mathcal{U}%
_{n}^{\alpha _{2}}}{\mathcal{V}_{n}^{\beta _{2}}}+(1-t_{n})\lambda |\mathcal{%
V}_{n}-\frac{\phi _{1}}{\hat{\theta}_{n}}| \\ 
\leq \frac{t_{n}}{\hat{\theta}_{n}^{1+\beta _{2}-\alpha _{2}}\rho ^{\beta
_{2}}}\mathcal{U}_{n}^{\alpha _{2}}+(1-t_{n})\lambda \mathcal{V}_{n}\leq 
\mathcal{U}_{n}^{\alpha _{2}}+\lambda \mathcal{V}_{n} \\ 
\leq \Vert \mathcal{U}_{n}\Vert _{C^{1}(\overline{\Omega })}^{\alpha
_{2}}+\lambda \Vert \mathcal{V}_{n}\Vert _{C^{1}(\overline{\Omega }%
)}=1+\lambda \text{ \ in }\Omega .%
\end{array}%
\end{equation}%
Then, owing to \cite{L} together with (\ref{25}), one derives that $\mathcal{%
U}{\normalsize _{n}}$ and $\mathcal{V}_{n}$ are bounded in $C^{1,\gamma }(%
\overline{\Omega })$ for certain $\gamma \in (0,1)$. The compactness of the
embedding $C^{1,\gamma }(\overline{\Omega })\subset C^{1}(\overline{\Omega })
$ implies 
\begin{equation*}
\mathcal{U}{\normalsize _{n}}\rightarrow \mathcal{U}\ \ \text{and \ }%
\mathcal{V}{\normalsize _{n}}\rightarrow \mathcal{V}\text{ \ in }C^{1}(%
\overline{\Omega }),
\end{equation*}%
with%
\begin{equation}
\begin{array}{l}
\mathcal{U},\mathcal{V}>\rho \text{ \ on }\overline{\Omega }.%
\end{array}
\label{5*}
\end{equation}%
Passing to the limit in (\ref{13*}) as $n\rightarrow \infty $ one gets%
\begin{equation}
\left\{ 
\begin{array}{ll}
-\Delta \mathcal{U}+\mathcal{U}=(1-t)\lambda \mathcal{U} & \text{in }\Omega 
\\ 
-\Delta \mathcal{V}+\mathcal{V}=(1-t)\lambda \mathcal{V} & \text{in }\Omega 
\\ 
\frac{\partial \mathcal{U}}{\partial \eta }=\frac{\partial \mathcal{V}}{%
\partial \eta }=0\text{ } & \text{on }\partial \Omega 
\end{array}%
\right. ,\text{ \ for }t\in \lbrack 0,1].  \label{6}
\end{equation}%
If $t=1$ then $\mathcal{U}=\mathcal{V}=0$ which contradicts (\ref{5*}).
Assume $t\in \lbrack 0,1).$ Since $(1-t)\lambda <\lambda <\lambda _{1},$ the
unique solution of (\ref{6}) is $(\mathcal{U},\mathcal{V})=0$, which is
impossible in view of (\ref{5*}). The claim is thus proved.
\end{proof}

\begin{proposition}
\label{P1}Assume (\ref{alphabeta}) holds.\ With a constant $L_{1}>0$, let
the homotopy $\mathcal{H}$ on $\left[ 0,1\right] \times \overline{\mathcal{B}%
}_{L_{1}}(0)$ defined by%
\begin{equation*}
\mathcal{H}(t,u,v)=(\mathcal{H}_{1}(t,u,v),\mathcal{H}_{2}(t,u,v)),
\end{equation*}%
with 
\begin{equation*}
\mathcal{H}_{1}(t,u,v)=u-(-\Delta +I)^{-1}(t\frac{u^{\alpha _{1}}}{v^{\beta
_{1}}}+(1-t)\lambda |u^{+}-\phi _{1}|),
\end{equation*}%
and%
\begin{equation*}
\mathcal{H}_{2}(t,u,v)=v-(-\Delta +I)^{-1}(t\frac{u^{\alpha _{2}}}{v^{\beta
_{2}}}+(1-t)\lambda |v^{+}-\phi _{1}|),
\end{equation*}%
where $\overline{\mathcal{B}}_{L_{1}}(0)$ is the closure of $\mathcal{B}%
_{L_{1}}(0)$ in $C^{1}(\overline{\Omega })\times C^{1}(\overline{\Omega })$.

If $L_{1}>0$ is sufficiently large, then the Leray-Schauder topological
degree $\deg (\mathcal{H}(t,\cdot ,\cdot ),\mathcal{B}_{L_{1}}(0),0)$ is
well defined for every $t\in \lbrack 0,1],$ and it holds%
\begin{equation}
\deg \left( \mathcal{H}(1,\cdot ,\cdot ),\mathcal{B}_{L_{1}}(0),0\right) =0.
\label{35}
\end{equation}
\end{proposition}

\begin{proof}
According to Lemma \ref{L1}, an analysis similar to the one showing (\ref{5}%
) implies that there exists a constant $\hat{\rho}>0$ such that all
solutions $(u,v)\in C^{1}(\overline{\Omega })\times C^{1}(\overline{\Omega }%
) $ of $(\mathrm{P}_{t})$ satisfy $u,v>\hat{\rho}$ on $\overline{\Omega }$,
for all $t\in \lbrack 0,1]$. Hence, the right hand-side of the equations in $%
(\mathrm{P}_{t})$ are continuous almost everywhere in $\Omega $ and
therefore, homotopies $\mathcal{H}_{1}$,$\mathcal{H}_{2}$ are well defined.
Moreover, $\mathcal{H}_{1},\mathcal{H}_{2}:\left[ 0,1\right] \times C^{1}(%
\overline{\Omega })\times C^{1}(\overline{\Omega })\rightarrow C(\overline{%
\Omega })$ are completely continuous. This is due to the compactness of the
operators $(-\Delta +I)^{-1}:C(\overline{\Omega })\rightarrow C^{1}(%
\overline{\Omega })$. Hence, $(u,v)\in \mathcal{B}_{L_{1}}(0)$ is a solution
for $(\mathrm{P})$ if, and only if, 
\begin{equation*}
\begin{array}{c}
(u,v)\in \mathcal{B}_{L_{1}}(0)\,\,\,\text{and}\,\,\,\mathcal{H}(1,u,v)=0.%
\end{array}%
\end{equation*}%
Clearly, the previous Lemma \ref{L1} ensures that solutions of $(\mathrm{P}%
_{t})$ lie in $\mathcal{B}_{L_{1}}(0)$. Furthermore, since $(\mathrm{P}_{t})$
has no solutions for $t=0$, one can deduce that%
\begin{equation*}
\deg \left( \mathcal{H}(0,\cdot ,\cdot ),\mathcal{B}_{L_{1}}(0),0\right) =0.
\end{equation*}%
Consequently, through the homotopy invariance property, we conclude that (%
\ref{35}) holds true.
\end{proof}

\subsubsection{\textbf{Topological degree on }$\mathcal{B}_{L_{2}}(0)$%
\textbf{.}}

Let us define the problem%
\begin{equation*}
(\mathrm{\tilde{P}}_{t})\qquad \left\{ 
\begin{array}{ll}
-\Delta u+u=t\frac{u^{\alpha _{1}}}{v^{\beta _{1}}}+(1-t)\lambda (u-\phi
_{1})^{+} & \text{in }\Omega , \\ 
-\Delta v+v=t\frac{u^{\alpha _{2}}}{v^{\beta _{2}}}+(1-t)\lambda (v-\phi
_{1})^{+} & \text{in }\Omega , \\ 
\frac{\partial u}{\partial \eta }=\frac{\partial v}{\partial \eta }=0\text{
on }\partial \Omega , & 
\end{array}%
\right.
\end{equation*}%
for $t\in \lbrack 0,1].$

\begin{lemma}
\label{L2}Under assumption (\ref{alphabeta}), every solution $(u,v)$ of $(%
\mathrm{\tilde{P}}_{t}),$ for $t\in \lbrack 0,1]$, is bounded in $C^{1}(%
\overline{\Omega })\times C^{1}(\overline{\Omega })$ and there exists a
constant $L_{2}>0$ such that%
\begin{equation}
\left\Vert u\right\Vert _{C^{1}(\overline{\Omega })},\left\Vert v\right\Vert
_{C^{1}(\overline{\Omega })}<L_{2}.  \label{28}
\end{equation}%
Moreover, problem $(\mathrm{\tilde{P}}_{0})$ admits only a trivial solution.
\end{lemma}

\begin{proof}
Arguing as in the proof of Lemma \ref{L1} we show that the solution set of
problem $(\mathrm{\tilde{P}}_{t})$ is bounded in $C^{1}(\overline{\Omega }%
)\times C^{1}(\overline{\Omega })$ uniformly with respect to $t\in \lbrack
0,1]$. Indeed, by contradiction, suppose that for every positive integer $n$%
, there exist $t_{n}\in \lbrack 0,1]$ and a solution $(u_{n},v_{n})$ of $(%
\mathrm{\tilde{P}}_{t_{n}})$ such that $t_{n}\rightarrow t\in \lbrack 0,1]$\
and $\Vert u_{n}\Vert _{C^{1}(\overline{\Omega })},\Vert v_{n}\Vert _{C^{1}(%
\overline{\Omega })}\rightarrow \infty $ as $n\rightarrow \infty .$
Observing that%
\begin{equation*}
(u_{n}-\phi _{1})^{+}\leq |u_{n}^{+}-\phi _{1}|\text{ \ and \ }(v_{n}-\phi
_{1})^{+}\leq |v_{n}^{+}-\phi _{1}|,
\end{equation*}%
the reasoning developed in the proof of Lemma \ref{L1} based on assumptions (%
\ref{alphabeta}) leads to the existence of a positive function $(\mathcal{%
\bar{U}},\mathcal{\bar{V}})\in C^{1}(\overline{\Omega })\times C^{1}(%
\overline{\Omega })$ solving the decoupled system (\ref{6}). This leads to
the same contradiction as in the proof of Lemma \ref{L1}, thereby completing
the proof of (\ref{28}).

Finally, for $t=0$, $(\mathrm{\tilde{P}}_{0})$ is expressed as a decoupled
system:%
\begin{equation*}
\left\{ 
\begin{array}{ll}
-\Delta u+u=\lambda (u-\phi _{1})^{+} & \text{in }\Omega , \\ 
-\Delta v+v=\lambda (v-\phi _{1})^{+} & \text{in }\Omega , \\ 
\frac{\partial u}{\partial \eta }=\frac{\partial v}{\partial \eta }=0\text{ }
& \text{on }\partial \Omega ,%
\end{array}%
\right.
\end{equation*}%
which, since $\lambda \in (0,\lambda _{1})$, admits only the trivial
solution $(u,v)=(0,0)$.
\end{proof}

\begin{proposition}
\label{P2}Assume (\ref{alphabeta}) holds.\ With a constant $L_{2}>0$, let
the homotopy $\mathcal{N}$ on $\left[ 0,1\right] \times \overline{\mathcal{B}%
}_{L_{2}}(0)$ defined by%
\begin{equation*}
\mathcal{N}(t,u,v)=(\mathcal{N}_{1}(t,u,v),\mathcal{N}_{2}(t,u,v)),
\end{equation*}%
with 
\begin{equation*}
\mathcal{N}_{1}(t,u,v)=u-(-\Delta +I)^{-1}(t\frac{u^{\alpha _{1}}}{v^{\beta
_{1}}}+(1-t)\lambda (u-\phi _{1})^{+}),
\end{equation*}%
and%
\begin{equation*}
\mathcal{N}_{2}(t,u,v)=v-(-\Delta +I)^{-1}(t\frac{u^{\alpha _{2}}}{v^{\beta
_{2}}}+(1-t)\lambda (v-\phi _{1})^{+}),
\end{equation*}

If $L_{2}>0$ is sufficiently large, then the Leray-Schauder topological
degree $\deg (\mathcal{N}(t,\cdot ,\cdot ),\mathcal{B}_{L_{2}}(0),0)$ is
well defined for every $t\in \lbrack 0,1],$ and it holds%
\begin{equation}
\deg \left( \mathcal{N}(1,\cdot ,\cdot ),\mathcal{B}_{L_{2}}(0),0\right) =1.
\label{36}
\end{equation}
\end{proposition}

\begin{proof}
By a similar argument to the one showing (\ref{5}), we infer that all
solutions $(u,v)\in C^{1}(\overline{\Omega })\times C^{1}(\overline{\Omega }%
) $ of $(\mathrm{\tilde{P}}_{t})$ satisfy $u,v>\hat{\rho}$ on $\overline{%
\Omega }$, for certain constant $\hat{\rho}>0,$ for all $t\in \lbrack 0,1]$.
Hence, due to the continuity of the right hand-side of the equations in $(%
\mathrm{P}_{t})$ for almost everywhere in $\Omega $, the homotopies $%
\mathcal{N}_{1}$,$\mathcal{N}_{2}$ are well defined. Moreover, from the
compactness of the operators $(-\Delta +I)^{-1}:C(\overline{\Omega }%
)\rightarrow C^{1}(\overline{\Omega })$ we derive that $\mathcal{N}_{1},%
\mathcal{N}_{2}:\left[ 0,1\right] \times C^{1}(\overline{\Omega })\times
C^{1}(\overline{\Omega })\rightarrow C(\overline{\Omega })$ are completely
continuous.

By the definition of $\mathcal{N}$, we infer that $(u,v)\in \mathcal{B}%
_{L_{2}}(0)$ is a solution of system $(\mathrm{P})$ if, and only if, 
\begin{equation*}
\begin{array}{c}
(u,v)\in \mathcal{B}_{L_{2}}(0)\,\,\,\text{and}\,\,\,\mathcal{N}(1,u,v)=0.%
\end{array}%
\end{equation*}%
In view of Lemma \ref{L2} solutions of $(\mathrm{\tilde{P}}_{t})$ lie in $%
\mathcal{B}_{L_{2}}(0)$ and%
\begin{equation*}
\deg \left( \mathcal{N}(0,\cdot ,\cdot ),\mathcal{B}_{L_{2}}(0),0\right) =1.
\end{equation*}%
Consequently, the homotopy invariance property leads to (\ref{36}).
\end{proof}

\subsubsection{\textbf{The degree on }$\mathcal{B}_{L_{1}}(0)\backslash 
\overline{\mathcal{B}}_{L_{2}}(0)$\textbf{.}}

\begin{proposition}
\label{P6}Assume that (\ref{alphabeta}) is satisfied. Let the map $\mathcal{M%
}:C^{1}(\overline{\Omega })\times C^{1}(\overline{\Omega })\rightarrow C(%
\overline{\Omega })$ be given by the following:%
\begin{equation*}
\mathcal{M}(u,v)=\left( u-(-\Delta +I)^{-1}(\frac{u^{\alpha _{1}}}{v^{\beta
_{1}}}),v-(-\Delta +I)^{-1}(\frac{u^{\alpha _{2}}}{v^{\beta _{2}}})\right) .
\end{equation*}%
For constants $L_{1},L_{2}>0$ with $L_{1}>L_{2}$, if the Leray-Schauder
topological degree $\deg (\mathcal{M}(\cdot ,\cdot ),\mathcal{B}%
_{L_{i}}(0),0)$ is well defined for $i=1,2$, then 
\begin{equation}
\deg (\mathcal{M}(\cdot ,\cdot ),\mathcal{B}_{L_{1}}(0)\backslash \overline{%
\mathcal{B}}_{L_{2}}(0),0)\neq 0.  \label{32}
\end{equation}
\end{proposition}

\begin{proof}
Fix $L_{2}>0$ such that the conclusion of Proposition \ref{P2} be satisfied
and choose $L_{1}>L_{2}$ so large to fulfil the conclusion of Proposition %
\ref{P1}. The weak comparison principle (\cite[Lemma 3.2]{ST}) applied to
problems $(\mathrm{P}_{t})$ and $(\mathrm{\tilde{P}}_{t})$ by making use of
the inequality $(s-\phi _{1})^{+}\leq |s^{+}-\phi _{1}|,$ for all $s\in 
\mathbb{R}
$, implies that the inclusion $\overline{\mathcal{B}}_{L_{2}}(0)\subset 
\mathcal{B}_{L_{1}}(0)$ holds. In view of the expressions of $\mathcal{M}$
and the homotopies $\mathcal{H}$, $\mathcal{N}$, it is seen that%
\begin{equation*}
\mathcal{H}(1,\cdot ,\cdot )=\mathcal{N}(1,\cdot ,\cdot )\,\,\,\text{in}%
\,\,\,\overline{\mathcal{B}}_{L_{2}}(0)
\end{equation*}%
and%
\begin{equation}
\mathcal{M}(\cdot ,\cdot )=\mathcal{H}(1,\cdot ,\cdot )\,\,\,\text{in}\,\,\,%
\mathcal{B}_{L_{1}}(0).  \label{31}
\end{equation}%
Moreover, since $\mathcal{H}(1,\cdot ,\cdot )$ and $\mathcal{N}(1,\cdot
,\cdot )$ do not vanish on $\partial \mathcal{B}_{L_{1}}(0)$ and $\partial 
\mathcal{B}_{L_{2}}(0)$, respectively, the Leray-Schauder topological degree
of $\mathcal{H}(1,\cdot ,\cdot )$ on $\mathcal{B}_{L_{1}}\backslash 
\overline{\mathcal{B}}_{L_{2}}$ makes sens. By the excision property of the
degree, we get%
\begin{equation*}
\begin{array}{l}
\deg (\mathcal{H}(1,\cdot ,\cdot ),\mathcal{B}_{L_{1}}(0),0)=\deg (\mathcal{H%
}(1,\cdot ,\cdot ),\mathcal{B}_{L_{1}}(0)\backslash \partial \mathcal{B}%
_{L_{2}}(0),0).%
\end{array}%
\end{equation*}%
By virtue of the domain additivity property of the degree it follows that 
\begin{equation*}
\begin{array}{c}
\deg (\mathcal{H}(1,\cdot ,\cdot ),\mathcal{B}_{L_{1}}(0)\backslash 
\overline{\mathcal{B}}_{L_{2}}(0),0)+\deg (\mathcal{H}(1,\cdot ,\cdot ),%
\mathcal{B}_{L_{2}}(0),0) \\ 
=\deg (\mathcal{H}(1,\cdot ,\cdot ),\mathcal{B}_{L_{1}}(0),0).%
\end{array}%
\end{equation*}%
Hence by (\ref{35}), (\ref{36}) and bearing in mind (\ref{31}), we deduce
that%
\begin{equation*}
\begin{array}{l}
\deg (\mathcal{M}(\cdot ,\cdot ),\mathcal{B}_{L_{1}}(0)\backslash \overline{%
\mathcal{B}}_{L_{2}}(0),0) \\ 
=\deg (\mathcal{H}(1,\cdot ,\cdot ),\mathcal{B}_{L_{1}}(0)\backslash 
\overline{\mathcal{B}}_{L_{2}}(0),0)=-1.%
\end{array}%
\end{equation*}%
This ends the proof.
\end{proof}

\subsection{Proof of the main result (Existence of a third solution)}

\begin{proof}[Proof of Theorem \protect\ref{T5}]
Two distinct positive solutions $(u_{1},v_{1})\in \lbrack C^{-1}\phi
_{1},Cz]\times \lbrack C^{-1}\phi _{1},Cz]$ and $(u_{2},v_{2})\in \lbrack
C^{-1}y_{\delta },Cy]\times \lbrack C^{-1}y_{\delta },Cy],$ belonging to $%
C^{1,\tau }(\overline{\Omega })\times C^{1,\tau }(\overline{\Omega })$ for
some $\tau \in (0,1),$ are obtained for system $(\mathrm{P})$ by Theorems %
\ref{T1} and \ref{T3}. The proof of Theorem \ref{T5} is completed by showing
that $(\mathrm{P})$ admits a positive solution $(u_{3},v_{3})\in C^{1,\tau }(%
\overline{\Omega })\times C^{1,\tau }(\overline{\Omega }),$ for certain $%
\tau \in (0,1)$, such that 
\begin{equation}
u_{3}\neq u_{i}\text{ \ and \ }v_{3}\neq v_{i}\text{, }i=1,2.  \label{7}
\end{equation}%
Fix 
\begin{equation*}
L_{2}>C\max \{\left\Vert z\right\Vert _{C^{1,\tau }(\overline{\Omega }%
)},\left\Vert y\right\Vert _{C^{1,\tau }(\overline{\Omega })}\}
\end{equation*}%
such that the conclusion of Proposition \ref{P2} be satisfied. As a result,
both positive solutions $(u_{1},v_{1})$ and $(u_{2},v_{2})$ lie in $\mathcal{%
B}_{L_{2}}(0)$. Now, we take $L_{1}>L_{2}$ such that the conclusion of
Proposition \ref{P1} be fulfilled. Here, without any loss of generality, we
may assume that the constant $L_{1}$ is sufficiently large so that the ball $%
\mathcal{B}_{L_{1}}(0)$ contains all $C^{1}$-bound solutions of $(\mathrm{P}%
) $. Otherwise, there would be an infinite number of solutions with $C^{1}$%
-regularity and the proof of Theorem \ref{T5} is therefore completed.

Thus, in view of Proposition \ref{P6} there exists $(u_{3},v_{3})\in C^{1}(%
\overline{\Omega })\times C^{1}(\overline{\Omega })$ satisfying $\mathcal{M}%
(u_{3},v_{3})=0$. This implies that the pair $(u_{3},v_{3})$ is a solution
of $(\mathrm{P})$. Since $(u_{3},v_{3})\in \mathcal{B}_{L_{1}}(0)\backslash 
\overline{\mathcal{B}}_{L_{2}}(0)$ and the ordered rectangles $[C^{-1}\phi
_{1},Cz]\times \lbrack C^{-1}\phi _{1},Cz]$ and $[C^{-1}y_{\delta
},Cy]\times \lbrack C^{-1}y_{\delta },Cy]$ are contained in the ball $%
\mathcal{B}_{L_{2}}(0)$, the assertion (\ref{7}) holds true and therefore $%
(u_{3},v_{3})$ is a third nontrivial positive solution of $(\mathrm{P})$.
The regularity theory (see \cite[Theorem 1.2]{L}) ensures that $%
(u_{3},v_{3})\in C^{1,\tau }(\overline{\Omega })\times C^{1,\tau }(\overline{%
\Omega })$ for some $\tau \in (0,1)$.
\end{proof}

\end{document}